\documentclass[1p]{elsarticle}

\usepackage{lineno,hyperref}

\usepackage{amssymb}
\usepackage{eucal}
\newcommand{\R}{{\Bbb R}}

\newcommand{\N}{{\Bbb N}}
\newcommand{\C}{{\Bbb C}}

\usepackage{amsmath}

\newtheorem{thm}{Theorem}
\newtheorem{lemma}[thm]{Lemma}
\newtheorem{corollary}[thm]{Corollary}
\newtheorem{proposition}[thm]{Proposition}

\newtheorem{remark}[thm]{Remark}
\newtheorem{example}[thm]{Example}
\newproof{proof}{Proof}

\begin{document}

\begin{frontmatter}

\title{Monotone waves for  non-monotone and non-local  monostable   reaction-diffusion equations} 

\author[a]{Elena Trofimchuk}
\author[b]{Manuel Pinto}
\author[c]{and Sergei Trofimchuk\footnote{Corresponding author.}}
\address[a]{Department of Mathematics II,
National Technical University, Kyiv,  Ukraine
\\ {\rm E-mail: trofimch@imath.kiev.ua}}
\address[b]{Facultad de Ciencias, Universidad
de Chile,  San\-tia\-go, Chile  
\\ {\rm E-mail: pintoj@uchile.cl}}
\address[c]{Instituto de Matem\'atica y F\'isica, Universidad de Talca, Casilla 747,
Talca, Chile \\ {\rm E-mail: trofimch@inst-mat.utalca.cl}}

\begin{abstract}
We propose a criterion for the existence of monotone wavefronts in non-monotone and  non-local  monostable  diffusive equations of the Mackey-Glass type.  This extends  recent results by Gomez {\it et al.} \cite{FGT} proved for  the particular case of equations with local delayed reaction. In addition, we demonstrate  the uniqueness (up to a translation) of obtained monotone wavefront within the class of all monotone wavefronts (such  a kind of conditional uniqueness was recently established  for the non-local KPP-Fisher equation by Fang and Zhao).  Moreover,  we show that if  delayed reaction  is local then this uniqueness actually holds within the class of all wavefronts  and therefore the minimal fronts  under consideration (either pulled or pushed) should be monotone.  Similarly to the case of  the  KPP-Fisher equations, 
our approach is based on the construction of an appropriate  fundamental solution for associated  boundary value problem for linear integral-differential equation. 
\end{abstract}


\end{frontmatter}


\section{Introduction and main results}

\paragraph{Introduction} In this work,  we  study the  existence and uniqueness of {\it monotone} wavefronts $u(x,t)= \phi(x+ct)$ for the monostable  delayed non-local reaction--diffusion equation 
\begin{equation}\label{17a} \hspace{0mm}
u_t(t,x) = u_{xx}(t,x)  - u(t,x) + \int_{\R}K(x-y)g(u(t-h,y))dy, \ u
\geq 0, 
\end{equation}
when the reaction term $g:\R_+\to\R_+$ neither is {\it monotone} nor defines a {\it quasi-monotone} functional in the sense of Wu-Zou \cite{wz} or Martin-Smith \cite{MS} and when the 
non-negative kernel $K(s)$  is Lebesgue integrable on $\R$.  Equation (\ref{17a}) is an important object of studies in the population dynamics, see \cite{BY,fzJDE,SEDY,gouss,LZii,ma1,MLLS,MeiI,tz,TAT,WLR,XX,YCW,YZ}. 
Taking formally $K(s)=\delta(s)$, the Dirac delta function,   we obtain the diffusive Mackey-Glass type equation 
\begin{equation}\label{17b} \hspace{0mm}
u_t(t,x) = u_{xx}(t,x)  - u(t,x) + g(u(t-h,x)), \ u
\geq 0,
\end{equation} 
another popular focus  of investigation, see \cite{FGT,IGT,TPT} for more details and references. 

 In the classical case,  when $h=0$, all 
wavefronts to the monostable equation (\ref{17b}) are monotone and, given a fixed admissible wave velocity $c$, all of them  are generated by a unique  front by means of   translations.  The same  monotonicity-uniqueness principle  is  valid  for certain subclasses of equations (\ref{17b}) with  $h >0$ (e.g. when $g$ is monotone  \cite{TPT}) and even  for equations (\ref{17a}) (e.g. when $g$ is a monotone and globally Lipschitzian function, with the Lipschitz constant $g'(0)$,  and when additionally $K(s)=K(-s), \ s \in \R$ \cite{ma1,tz}).  
However, if the reaction term $g$ is non-monotone and non-local,  monotonicity and uniqueness are not longer obligatory front's characteristics.  For example,  \cite{TAT} provides conditions sufficient for  non-monotonicity of wavefronts' profiles for non-local equation (\ref{17a}) with compactly supported kernel $K$.  Co-existence of multiple wavefronts for non-local models is also 
known from  \cite{HK,NPTT}. 
Thus the  question about  the existence and uniqueness of {\it monotone} wavefronts  for the monostable  {\it non-monotone} non-local (or delayed) reaction-diffusion equations  seems to be interesting and timely. In fact,  
recently it called the attention of several researchers. In this regard, the most studied model was the non-local 
KPP-Fisher equation \cite{ZAMP,BNPR,FZ,HK,NPT,NPTT}
\begin{equation}\label{17nl}
u_t(t,x) = u_{xx}(t,x)  + u(t,x)(1-\int_\R K(x-y)u(t,y)dy),
\end{equation}
and its local delayed version \cite{ZAMP,ADN,KO,HTa,GT,FGT,wz} (called  the
diffusive Hutchinson's equation)  
\begin{equation}\label{17} \hspace{5mm}
u_t(t,x) = u_{xx}(t,x)  + u(t,x)(1-u(t-\tau,x)). \end{equation} 
The above cited papers elaborated a complete characterization of models (\ref{17nl}) and (\ref{17}) possessing monotone wavefronts. Moreover, the absolute uniqueness (i.e. uniqueness within the class of all wavefronts) of monotone 
wavefronts to (\ref{17}) and the conditional uniqueness (i.e. uniqueness within the subclass of monotone wavefronts) of monotone 
wavefronts to (\ref{17nl}) was also proved in these works. As we have mentioned, in general, monotone and non-monotone wavefronts can coexist in  (\ref{17nl})  \cite{HK,NPTT}. 

In the case of models (\ref{17a}) and (\ref{17b}) having non-monotone  function $g$,  the  existence of {\it monotone} wavefronts was analyzed only for equation (\ref{17b}) in \cite{FGT}, by the help of the Hale-Lin functional-analytic approach and a continuation argument.  This  method required a detailed
analysis of  a family of  linear differential Fredholm
operators associated with (\ref{17b}).  The discrete Lyapunov functionals of Mallet-Paret and Sell for delayed differential equations were also used  in an essential way. Therefore the task of extension of the approach 
developed in  \cite{FGT} on  non-local equations (\ref{17a})  seems to be quite difficult (if anyhow possible).  Consequently,  the main goal of the present paper is to provide an alternative  
technique allowing to analyse  monotonicity  of wavefronts for non-monotone and non-local  equation (\ref{17a}).  A key feature  of this technique consists in reduction of the wave profile equation for  (\ref{17a})  to a new and non-obvious  convolution equation (see Section \ref{FH}).  The obtained nonlinear  equation is then studied by means of various already established methods (in particular, we use the Berestycki-Nirenberg sliding solution method  as well as  approaches developed in \cite{AGT,SEDY}). 
Remarkably, by weakening restrictions imposed on the birth function $g$,  our technique  improves  the existence criterion of  \cite{FGT}.  Furthermore, it also enables us to prove  the conditional uniqueness of monotone fronts for equations  (\ref{17a}), (\ref{17b}). Note that for non-monotone equations  (\ref{17a}), (\ref{17b}), the uniqueness  problem was only partially  solved in \cite{AGT,fzJDE,FGT,tz} by means of  the Diekmann-Kaper method \cite{AGT,dk}. This, however,  presumes the Lipschitz condition  $|g(s_1)-g(s_2)|\leq g'(0)|s_1-s_2|$ which is sufficient for the absolute uniqueness of each front (either monotone or non-monotone) but excludes from the consideration so-called pushed fronts. These fronts are quite significant  from the ecological point of view, see \cite{STR,TPT} and references therein. 
Our results (Theorems \ref{main2} and \ref{mainPF} below) provide a simple criterion guaranteeing  the absolute uniqueness of monotone wavefronts for (\ref{17b}), including the minimal one
(it doesn't matter whether it is pushed or not).  
It should be noted here that, as the results of \cite{IGT, FGT,TT}  show, in many cases only the slow traveling  fronts of  (\ref{17b}) can be monotone (or eventually monotone) while the increase of  propagation speed might lead to the appearance of  slowly oscillating waves.  This is why the studies of monotonicity of the slow
 traveling fronts to equations (\ref{17a}), (\ref{17b}) is of our primary interest.  On the other hand, 
already starting from the seminal work of Kolmogorov, Petrovskii and Piskunov \cite{KPP},  
the applied importance of the slowest (i.e. minimal or critical) waves  is also well known: in particular, it reflects the fact that the invasion of a new unexplored territory by a single species population is realised with the so-called minimal speed of propagation $c_*$, in the form of the minimal  front $u(x,t)= \phi_*(x+c_*t)$.  Clearly, the geometric properties of the leading edge of  the invasion profile $\phi_*(s),\ s\in \R,$  are quite significant for the description of   transition  from an uninvaded  to invaded state. 

\vspace{1mm}

In the subsequent part of this section, we state the key hypotheses used in the paper  and briefly discuss our main theorems together with a key auxiliary  assertion.   

\vspace{1mm}

\paragraph{Main assumptions}\

\vspace{1mm}

\noindent  {\rm \bf(M)}  $g\in C(\R_+), \ g(s)>0$ for $s>0$,  and the equation $g(s)= s$ has exactly two nonnegative solutions:
$0$ and $\kappa >0$. Moreover, $g$ is differentiable at the equilibria  with $g'(0) >1$ and $g'(\kappa) <0$. 

\vspace{2mm}

\noindent  {\rm \bf(ST)}  $g(s)- g'(\kappa)s \ \mbox{is non-decreasing on } \ [0,\kappa]. 
$
Observe that the last assumption implies the sub-tangency property of $g$ at $\kappa$: $g(s) \leq g(\kappa) + g'(\kappa)(s-\kappa),\ s \in [0,\kappa].$ 

\vspace{2mm}

\noindent  {\rm \bf(K)}  $K\geq 0$ and  $\int_\R K(s)ds=1$. Moreover,  $\int_\R K(s)e^{-\lambda s}ds <\infty$ for 
each $\lambda \in \R$.

\begin{example} If $g$ is differentiable on $[0,\kappa]$ then the hypothesis {\rm \bf(ST)}  amounts to the inequality  {\rm \bf(ST$'$)}:  $g'(s) \geq g'(\kappa)$ satisfied for all  $s \in [0,\kappa]$. For the following non-local version of the popular Nicholson's blowflies diffusive equation 
$$
u_t(t,x) = u_{xx}(t,x)  - \delta u(t,x) + p\int_{\R}K(x-y)u(t-h,y)e^{-u(t-h,y)}dy,  \quad p>\delta >0, 
$$
the  assumptions  {\rm \bf(M)}   and  {\rm \bf(ST$'$)} are equivalent to the inequalities $e < p/\delta \leq e^2$. We note that the bulk of information concerning the Nicholson's diffusive equation is obtained  for a simpler (monotone) case when  $1< p/\delta \leq e$ (cf. the recent works \cite{XX} and \cite{STR-2}). 
\end{example}

\paragraph{Main results: existence} Clearly, $u(t,x)= \phi(x+ct)$ is a front solution of equation  (\ref{17a})  if and only if the profile $y=\phi(t)$ solves the boundary value problem
\begin{equation}\label{yp}
y''(t) -cy'(t)-y(t) + \int_\R K(t-s) g(y(s-ch))=0, \ y(-\infty)=0, \ y(+\infty)=\kappa, \ y(t) \geq 0. 
\end{equation}
For kernel $K$ satisfying  {\rm \bf(K)}, we  will consider the characteristic 
functions 
$$
\chi_0(z)=z^2-cz-1+g'(0) e^{-chz}\int_\R e^{-zs}K(s)ds, \quad z \in \C,$$
$$
\chi_\kappa(z)=z^2-cz-1+ g'(\kappa)e^{-chz}\int_\R e^{-zs}K(s)ds, \quad z \in \C,
$$
associated with the linearizations of (\ref{yp}) at the equilibria $0$ and $\kappa$, respectively. 
We will need the following  three subsets $\mathcal{D}_{0}, \mathcal{D}_{\kappa}, \mathcal{D}_{\frak L} := \overline{\mathcal{D}}_{0} \cap \mathcal{D}_{\kappa}$ of the half-plane $(h,c)\in \R_+\times \R$:
$$
\mathcal{D}_{\kappa}= \left\{(h,c)\in \R_+\times \R:  \chi_\kappa(z) \ \mbox{has at least one positive and one negative simple zeros}\right\}; 
$$ 
$$
\mathcal{D}_{0}= \left\{(h,c)\in \R_+\times \R:  \chi_0(z) \ \mbox{has exactly two positive zeros $\mu_0 <\mu_1$}\right\}.$$
The geometric description  of the open domain $\mathcal{D}_{0}$ is well known and it is  summarised in the following assertion:
\begin{proposition} Assume that $g'(0)>1$. Then for each $h \geq 0$ there exists a unique  $c=c_\#(h)\in \R$ such that the equation $\chi_0(z)=0$ with this $c$ has a unique positive double root. The 
function $c_\#:\R_+ \to \R$ 
is $C^\infty$-continuous and strictly decreasing. Furthermore, 
$\mathcal{D}_{0}$ coincides with the set $\left\{(h,c)\in \R_+\times \R:  c > c_\#(h)\right\}$. \end{proposition}
\begin{proof} For example, see \cite[Lemma 22]{SEDY} and  \cite[Theorem 1.1]{AV}. By \cite{AV} if, in addition, 
$K(s)=K(-s), \ s \in \R,$ then $0< c_\#(h)=O(1/h)$ at $+\infty$. In general, however, $c_\#(h)$ can take negative values, cf. \cite{SEDY}. \hfill $\square$
\end{proof}
It is clear that the finite part of the boundary of $\mathcal{D}_{\frak L}$  consists from the curves 
determined either from the system $\chi_0(z)=0, \ \chi'_0(z)=0$ or from the system $\chi_\kappa(z)=0, \ \chi'_\kappa(z)=0$.
Thus, in each particular case,  the shape of the domains $\mathcal{D}_{\kappa},\ \mathcal{D}_{\frak L}$ can be easily  identified. For instance, if $K(s)$ is the Dirac's delta, then  $\mathcal{D}_{\frak L}$ is a simply connected domain whose boundary contains a non-empty segment of the half-line  $\{h=0, c \geq 0\}$ \cite{FGT}: 
$$ \mathcal{D}_{\frak L} = \{(h,c): h \in [0,h_*], \ h_* \leq +\infty,  \   0< c_\#(h) 
\leq c < c^*(h)\}, $$
where the smooth decreasing function $c^*(h)$ is determined by  $\chi_\kappa(z)$.
The aforementioned characteristics  of  $\mathcal{D}_{\frak L}$ are essential for the use of a continuation argument in \cite{FGT}.  For general kernels $K$, however, 
 the  set  $\mathcal{D}_{\frak L}$ eventually might be more complicated (for instance, not connected). 
One of  advantages of our present approach is that  it does not require any connectedness property from  $\mathcal{D}_{\frak L}$:
\begin{thm} \label{mai1}  
Assume  {\rm \bf(M)},  {\rm \bf(K)},  {\rm \bf(ST)} and that $g$ is sub-tangential at the equilibrium $0$: 
$
g(s) \leq g'(0)s, \ \mbox{for all } \ s \in [0,\kappa]$.  Then for each point $(h,c)$ in the closure $\overline{\mathcal{D}}_{\frak L}$ of the set  $\mathcal{D}_{\frak L}$,  
equation (\ref{17a}) has at least one wavefront $u(t,x)= \phi_c(x
+ct)$ with strictly increasing profile $\phi_c(s), \ s \in \R$.
\end{thm}
As we have  mentioned, the conclusion and the proof of Theorem \ref{mai1} are 
 also valid when $K(s)$ is  the Dirac delta function, i.e. 
 for the local equation  (\ref{17b}).  Thus it is enlightening to compare criterion of front's monotonicity for (\ref{17b}) established in   \cite[Theorem 2.2]{FGT} and  Theorem \ref{mai1}.  These two results almost coincide except for  two important details: $g$ in \cite{FGT} must be more smooth 
 ($C^{1,\gamma}-$continuous  on $[0,\kappa]$)  and  must have a unique critical
point on $(0,\kappa)$. That is,    the unimodal form of $g$ is assumed in \cite{FGT}  instead of the condition {\rm \bf(ST)}. Clearly, even if these both  requirements are fulfilled for the classical population models (Nicholson's blowflies model, hematopoiesis model), they are independent: so that Theorems \ref{mai1} and  \cite[Theorem 2.2]{FGT}  complement each other in the case of local delayed equations.   This comparison also shows that for some models  
(e.g. for equation (\ref{17b})) Theorem \ref{mai1} provides the necessary and sufficient conditions  of front's monotonicity. Indeed, it follows from 
\cite{FGT} that  (\ref{17b}) can not have any monotone front if  $(h,c)$ does not belong to  the closure of the set  $\mathcal{D}_{\frak L}$.

\paragraph{Main results: uniqueness} To prove the conditional uniqueness (up to a translation) of a monotone wavefront,  instead of  the above mentioned  Diekmann-Kaper theory \cite{AGT,dk,XX}, here   we are using 
an alternative approach based on  the sliding solution method developed by Berestycki and Nirenberg, cf.  \cite{TPT}.  This technique was successfully applied in \cite{chen,co,cdm, mazou} to prove the uniqueness of {\it monotone} wavefronts without  imposing  any Lipshitz condition on $g$.  In the present paper, we consider  sliding solutions to prove the following.
\begin{thm}\label{main2} Assume   {\rm \bf(M)}, {\rm \bf(K)} and {\rm \bf(ST)}. In addition,  let $g$ be $C^1$-smooth in some neighborhood of $\kappa$ and there exist $C >0,\ \theta \in (0,1],$ $\delta >0$ such that   
\begin{equation}\label{gco}
\left|g(u)/u- g'(0)\right| \leq Cu^\theta, \quad u\in  (0,\delta].
\end{equation}
Fix some $(c,h)\in  \mathcal{D}_{\frak L}$, and suppose that $u_1(t,x)= \phi(x+ct)$, $u_2(t,x)= \psi(x+ct)$ 
are two monotone traveling fronts of equation  (\ref{17a}). Then $\phi(s)= \psi(s+s_0), \ s \in \R,$ for some $s_0$. 
\end{thm}

\begin{remark}\label{R1} As a by-product of the proofs of Theorems \ref{mai1},  \ref{main2}, we obtain the following: 
Assume {\rm \bf(K)} and  {\rm \bf(M)} where $g'(\kappa)\geq 0$ is considered instead of $g'(\kappa)< 0$. If, in addition, $ 
g(s) \leq g'(0)s, \ s \in [0,\kappa]$, and
$g$ is monotone and  satisfies the smoothness conditions of Theorem \ref{main2}, then for each point $(h,c)$ in the closure of the set  $\mathcal{D}_{0}$,  
equation (\ref{17a}) has a unique (up to a translation) monotone wavefront $u(t,x)= \phi_c(x
+ct).$
\end{remark}

We will say that some $c$ is an admissible speed of propagation for (\ref{17a})  (or for (\ref{17b})) if there exists a positive wave solution $u=\phi(x+ct)$ to 
(\ref{17a})  (to (\ref{17b}), respectively) such that $\phi(-\infty)=0$ and $\liminf_{t\to +\infty} \phi(t) >0$. We call such a wave solution semi-wavefront.   As \cite{HK,NPTT} reveals, proper semi-wavefronts and monotone fronts can co-exist in non-local monostable equations. Nevertheless, as 
the next result shows, the statement of Theorem \ref{main2} can be strengthened for the case of local delayed reaction:
\begin{thm} \label{mainPF}Assume   {\rm \bf(M)}, {\rm \bf(K)} and {\rm \bf(ST)}. Then each semi-wavefront  $u = \phi(x+ct)$, $(h,c)\in {\mathcal{D}}_{\frak L},$ for equation (\ref{17b}) is actually a monotone front and therefore it is the unique possible wavefront solution of (\ref{17b}) (up to a translation) propagating with the speed $c$. 
\end{thm}
\begin{remark} Fix some $h>  0$ and consider 
$$
{\mathcal{C}}(h) := \{ c \geq 0: \ \mbox{there is  a semi-wavefront for (\ref{17b}) propagating at the velocity\ } c \}.
$$
It is well known (cf. \cite[Theorem 4]{TAT} or \cite[Remark 8]{IGT}) that ${\mathcal{C}}(h) $ contains some infinite subinterval $[c_1(h), +\infty)$ while 
$
c_*(h): =\inf {\mathcal{C}}(h) \geq c_\#(h) \in \R. 
$
It is easy to see that ${\mathcal{C}}(h)$ is closed (cf. \cite[Lemma 26]{GT}) so that $c_*(h) \in {\mathcal{C}}(h)$. 
The number $c_*(h)$ is called the minimal speed of propagation for the monostable model (\ref{17b}). In general, $c_*(h)$ is not {\it linearly determined}, i.e.  ${\mathcal{C}}(h) \not= [ c_\#(h), +\infty)$ \cite{HuM, XX}. The minimal wave is called pushed if  $c_*(h)> c_\#(h)$.   For equation (\ref{17b}) with monotone $g$, a min-max representation for the speed of pushed wavefront can be found in \cite{STR}.

For some models (for example, if $g$ is monotone on $[0,\kappa]$ \cite{LZii,TPT} or if $g$ is sub-tangential at $0$), we have  ${\mathcal{C}}(h)= [c_*(h),+\infty)$.  In addition, if  $g$ is also sub-tangential at $\kappa$, then ${\mathcal{C}}(h)$ for equation (\ref{17b}) can 
be represented as a union of three adjacent intervals  ${\mathcal{C}}(h)= \mathcal{I}_m\cup \mathcal{I}_{so}\cup\mathcal{I}_{sw}$ (some of them can be empty) corresponding to 
velocities of  monotone fronts, slowly oscillating fronts and proper semi-wavefronts, respectively  \cite{FGT,TT}.  If   $g$ is not sub-tangential at $\kappa$, then also non-monotone non-oscillating fronts can appear \cite{IGT}. In 
the important case, when 
$g$ is neither monotone on $[0,\kappa]$ nor sub-tangential at $0$,  the question about the connectedness 
of the set ${\mathcal{C}}(h)$ is largely open at this point of investigation. 
\end{remark}

\paragraph{An auxiliary result} It is convenient to transform the profile equation (\ref{yp})  into a suitable nonlinear convolution equation \cite{dk,wz}
\begin{equation}\label{ftc}
\phi(t)= \int_\R N(t-s) g_1(\phi(s-ch))ds, 
\end{equation}
with appropriate kernel $N \in L_1(\R,\R_+)$ and continuous monostable nonlinearity $g_1:\R_+\to \R_+$. Certainly, $N, g_1$ depend on $c,h, K, g$ and the choice of specific $N, g_1$ depends on the goals of investigation. A correct 
determination of $N$ and $g_1$ may indicate a shortest way in establishing various properties of profiles (including  their existence).  For instance, all above mentioned front's monotonicity criteria for the  
KPP-Fisher equations (\ref{17nl}) and (\ref{17}) were obtained after discovering a satisfactory form of the associated  convolution equation, see \cite{FZ,GT,KO}. Similarly, an important part of this paper is focused 
on reducing equation (\ref{yp}) to the `optimal' convolution equation:  
\begin{thm} \label{ar}  
Assume  {\rm \bf(M)} and   {\rm \bf(K)}.  Then for each point $(h,c)\in \mathcal{D}_{\frak L}$,  there exist $g_1$, 
positive $\delta$ and kernels $N, - v>0$ given by
$$N=-(1+\xi) K*v:= -(1+\xi) \int_\R K(s)v(t-s)ds, \
g_1(s)= \frac{g(s)+\xi s}{1+\xi}, \ \xi:= |g'(\kappa)|+\delta, $$
such that 
the boundary value problem (\ref{yp}) has a solution if and only if 
equation (\ref{ftc}) has a non-negative solution satisfying the boundary conditions of (\ref{yp}).  Furthermore, 
$\int_\R N(s)ds=1$ and  $\int_\R N(s)e^{-\lambda s}ds <\infty$ for all $\lambda$ from some maximal  finite 
interval $(\gamma_l, \gamma_r) \ni \{0\}$. Continuous  function  $v$ is $C^\infty$- smooth on $\R_-$ and $\R_+$ and has a unique minimum point at $t=0$. In fact, $v$ is strictly monotone on $\R_-$ and  $\R_+$ and it is  strictly convex on $\R_-$. 
\end{thm}
\begin{remark}\label{Re7} For equation  (\ref{17b}),  a more explicit  form of $N(t)$ can be obtained:  $$ N(t)= -(1+\xi)v(t,\xi),\ \mbox{ where} \ 
v(t,\xi) = -
 \frac{1}{\chi'(\lambda_0(\xi))} \left\{\begin{array}{cc} \tilde u(t), & t\geq 0, \\ e^{\lambda_0(\xi) t }, & t<0, \end{array}\right.
$$
 $\chi(z)=z^2-cz-1-\xi e^{-chz}$,  $\lambda_0(\xi)$ is the unique positive zero of  $\chi(z)$, and $\tilde u(t)$ is the solution of the following initial value problem:
\begin{equation}\label{red}
u''(t) -cu'(t)-u(t) -\xi u(t-ch) =0,
\end{equation}
$$ u(s) = e^{\lambda_0(\xi) s}, \ s \in [-ch,0], \quad  u'(0) = -(\lambda_0(\xi)-c+\xi che^{-\lambda_0(\xi)ch}). 
$$
When $\xi =0$, this formula for the fundamental solution $v(t,\xi) $ for (\ref{red}) is well known, see  (\ref{wn}).  It should be noted that  the explicit exponential form of $v(t)$ for negative $t$ will allow  to prove the monotonicity of all wavefronts under conditions of Theorem \ref{mainPF}.  On the other hand, the one-sided Laplace transform $\hat v(z)=\int_0^{+\infty}e^{-zt}v(t)dt$ of $ v(t)$ can also be easily found: 
$$
\hat v(z)= \frac{1}{\chi(z)} -\frac{1}{\chi'(\lambda_0)}\frac{1}{z-\lambda_0(\xi)}. 
$$
This function is analytic in the half-plane $\{\Re z > \lambda_1(\xi)\}$ where  $\lambda_1(\xi)$ is the biggest negative zero of  $\chi(z)$. As the Laplace transform of the negative function, $-\hat v(x)$, $x \in (\lambda_1(\xi), +\infty)$, provides a new example of completely monotone elementary function, see  \cite{AP,MR,WID}.

\end{remark}
Since the analytical properties of functions $g_1, v$ defined in Theorem \ref{ar} are rather nice (for example, $g_1$ is monotone if   {\rm \bf(ST)} is assumed), the optimality  of their choice for solving our existence/uniqueness problems has to be explained in terms of optimality  of  the set  $\mathcal{D}_{\frak L} = \overline{\mathcal{D}_{0}} \cap \mathcal{D}_{\kappa}$.  In the ideal case, the closure of $\mathcal{D}_{\frak L}$ must contain the set $\frak{M}$ of all pairs $(h,c)$
for which (\ref{yp}) has a unique (up to a shift) monotone solution.  Since it is well known (cf. \cite{SEDY}) that  $\frak{M} \subseteq  \overline{\mathcal{D}_{0}}$, we actually need only to justify   the choice of 
$$
\mathcal{D}_{\kappa}= \left\{(h,c)\in \R_+\times \R:  \chi_\kappa(z) \ \mbox{has at least one positive and one negative simple root}\right\}.
$$  
The necessity of the presence of at least one negative zero of  $\chi_\kappa(z)$ in the definition of $\mathcal{D}_{\kappa}$ for the existence of monotone fronts was proved both for the delayed equation (\ref{17b}) (e.g. see \cite{FGT}) and the non-local equation  (\ref{17a}) (at least when $K$ has compact support, cf. \cite[Theorem 6]{TAT}).  However, the necessity of the presence of at least one positive zero of  $\chi_\kappa(z)$ in the definition of $\mathcal{D}_{\kappa}$ is not so obvious. Here we give the following two arguments in support of  this presence: 
firstly, in some situations (e.g. for equation (\ref{17b})),  at least one positive zero of  $\chi_\kappa(z)$  exists automatically for all 
parameters in $\overline{\mathcal{D}_{0}}$; secondly,  a unique positive zero of $\chi_\kappa(z)$  plays an essential
role  in proof of  the monotonicity criterion  in \cite{FGT} (more precisely, in proofs of surjectivity of associated Fredholm operators, see  Proposition 3.2 and Lemma 3.3 in \cite{FGT}).

Finally, a few words about the organisation of the paper. In the next section, we define and study  properties
of the fundamental solution $v(t,\xi)$. These studies are resumed in Theorem \ref{ar} which is formally proved  in  
Step I  of Sections \ref{S3}. 
The convolution equation (\ref{ftc})  is then used to prove 
Theorems \ref{mai1}, \ref{main2} in Sections \ref{S3} and \ref{S4}, respectively. The proof of Theorem \ref{mainPF} is given in Section \ref{S5}.

\section{Negativity of the fundamental solution.}\label{FH}
\subsection{The fundamental solution: definitions and properties}
\noindent Fix $c,d,\xi \in \R$,    kernel $K(s)$ satisfying {\rm \bf(K)}  and consider the linear  integral-differential  inhomogeneous equation 
\begin{equation}\label{leq}
y''(t) -cy'(t)-dy(t) -\xi \int_\R K(t-s) y(s-ch)ds + f(t) =0,
\end{equation}
where $f:\R \to \R$ is a bounded continuous function and the characteristic function 
$$
\chi(z,\xi)=z^2-cz-d-\xi e^{-chz}\int_\R e^{-zs}K(s)ds, \quad z \in \C,$$
does not have zeros on the imaginary axis (in such a case, we will say that equation (\ref{leq}) is hyperbolic). 
Suppose, for a moment, that $f$ is compactly supported
and that, for this inhomogeneity,  equation  (\ref{leq}) has a solution $y:\R \to \R$ exponentially decaying, together with its first derivative $y'(t)$,  at $\pm\infty$. 
Then, 
applying the bilateral Laplace transformation to (\ref{leq}), 
we find easily that this equation has a solution 
$y(t) = -v*f(s)$,  which is the convolution of $f$ with the  bilateral Laplace inverse $v(t,\xi)$ of  
$1/\chi(\lambda,\xi)$.
Since $y(t)$ is a bounded function,   the formula $y(t) = -v*f(s)$ shows that the  inverse Laplace 
transform should be  applied to $1/\chi(\lambda,\xi)$ considered on the maximal vertical analyticity strip $\Pi(\lambda_l.\lambda_r):=\{z: \lambda_l < \Re z < \lambda_r \}$ that includes the imaginary axis (observe that $\lambda_l <0< \lambda_r$ since  
the imaginary axis does not contain any singular point of $1/\chi(\lambda,\xi)$). 
The function  $v(\cdot,\xi):\R\to \C$ is called the fundamental solution for equation  (\ref{leq}).  The above said and the inversion theorem  imply that 
\begin{equation}\label{ev}
v(t,\xi) = -\frac{1}{2\pi }\int_{-\infty}^{+\infty} \frac{e^{iut}du}{u^2+ci u+d +\xi e^{-iu ch}\int_\R K(s)e^{-ius}ds}, \ t \in \R.
\end{equation}
We view this formula as a formal definition of the fundamental solution for equation  (\ref{leq}). 

\vspace{-1mm}

\begin{lemma} \label{Le1}  Suppose that $\chi(z,\xi)$ does not have pure imaginary zeros. Then $v(\cdot,\xi):\R\to \R$ is a real valued function which is  infinitely differentiable with respect to $t $ on the set $\R\setminus\{0\}$ where it also satisfies equation
(\ref{leq}) with $f(t)\equiv 0$.  Next, there exist the limits $v'(0^-,\xi), v'(0^+,\xi)$ and $v'(0^+,\xi) - v'(0^-,\xi)=1$
(thus the limits $v''(0^-,\xi), v''(0^+,\xi)$ exist and $v''(0^+,\xi) - v''(0^-,\xi)=c$). 
\end{lemma}

\vspace{-2mm}

\begin{proof}
Indeed, $v$ is a real valued function because of the presentation
\begin{equation}\label{V0}
v(t,\xi) = -\frac{1}{\pi }\int_{0}^{+\infty} \frac{p(u)\cos(tu) + q(u) \sin (tu)}{p^2(u)+q^2(u)}du,\ t \in \R, 
\end{equation}
where  $p, q$ satisfying $p(u)=p(-u), q(u)=-q(-u), u \in \R$, are defined by 
$$
p(u):= u^2+d+\xi C(u), \ q(u):= cu-\xi S(u),
$$ 
where, due to the Lebesgue-Riemann lemma,  the functions
$$C(u): =\int_{\R} K(s) \cos(u(ch+s))ds, \quad S(u):= \int_{\R} K(s) \sin(u(ch+s))ds, $$
are vanishing, together with their derivatives of all orders, at $\infty$. In particular, this implies 
that 
$$
P(+\infty)=0, \ \ \mbox{where}\  \ P(u):=\frac{up(u)}{p^2(u)+q^2(u)},
$$
while derivatives of all orders $k=1,2,3,\dots$, 
$$
P^{(k)}(u)= (u^{-1})^{(k)}(1 + o(1)) = (-1)^k k!u^{-k-1} (1 + o(1)), \ u \to +\infty,  \ 
$$
are monotone at $+\infty$. Therefore, by  the Dirichlet test of the uniform convergence of improper integrals \cite[p. 421]{VAZ},  the integral
$$
\frac{1}{\pi }\int_{0}^{+\infty} \frac{up(u)\sin(tu) }{p^2(u)+q^2(u)}du
$$
converges uniformly on each compact subset of $\R\setminus\{0\}$. In consequence, (cf. \cite[p. 426]{VAZ}), 
\begin{equation}\label{V1}
v'(t,\xi) = \frac{1}{\pi }\int_{0}^{+\infty} \frac{up(u)\sin(tu) -uq(u) \cos(tu)}{p^2(u)+q^2(u)}du,\ t \not= 0, 
\end{equation}
exists for all $t\not=0$. Note that the term $uq(u)/(p^2(u)+q^2(u))$ is Lebesgue integrable on $\R_+$ so that 
 the function 
$$
I_2(t)= \int_{0}^{+\infty} \frac{uq(u) \cos(tu)}{p^2(u)+q^2(u)}du
$$
is continuous on $\R$. Hence,  in order to prove the existence of $v'(0^+,\xi), v'(0^-,\xi)$, we only need to take into account the  
integral 
$$
I_1(t):= \int_{0}^{+\infty} P_0(u)\frac{\sin(tu)}{u}du= \int_{0}^{+\infty} (1-P_1(u))\frac{\sin(tu)}{u}du= \frac{\pi\, {\rm sign}\, t}{2}- 
\int_{0}^{+\infty} P_1(u)\frac{\sin(tu)}{u}du
$$
where  $t \not= 0$, $|u^{-1}\sin (tu)| \leq |t|, \ u >0,$ and 
$$
P_0(u):= \frac{u^2p(u)}{p^2(u)+q^2(u)}, \quad
 P_1(u)= \frac{p(u)(1+\xi C(u)) + q^2(u)}{p^2(u)+q^2(u)} \in L_1(\R_+).$$
 Thus $I_1(0^+) = \pi/2, \ I_1(0^-) = -\pi/2,$ so that $v'(0^+,\xi) - v'(0^-,\xi)=1$. 
 
 Finally, in view of formulas (\ref{V0}), (\ref{V1}), a direct computation gives, for $t\not=0$, 
$$
v'(t,\xi)-cv(t,\xi) -\int_0^tv(s,\xi)ds -\xi\int_0^tdu\int_{\R}K(s)v(u-s-ch,\xi)ds= 
$$
$$
\frac{1}{\pi }\int_{0}^{+\infty} \frac{\left(p^2(u)+q^2(u)\right)\sin(tu) +\left(q(u)p(u) -p(u)q(u) \right)\cos(tu)}{u(p^2(u)+q^2(u))}du+
$$
$$
\frac{1}{\pi }\int_{0}^{+\infty} \frac{\xi p(u)S(u)+\xi q(u)C(u)+dq(u)}{u(p^2(u)+q^2(u))}du= \frac 1 2 + \frac{1}{\pi }\int_{0}^{+\infty} \frac{\xi p(u)S(u)+\xi q(u)C(u)+d q(u)}{u(p^2(u)+q^2(u))}du.
$$
Consequently, $v''(t,\xi)$ exists for $t\not=0$ and $v(t,\xi)$ satisfies equation (\ref{leq}) with $f(t)\equiv 0$ for all $t\not=0$. 
In fact,   if  $t>0$ (the case $t<0$ is similar) then 
 $$
v'(t,\xi) = \frac{1}{t\pi }\int_{0}^{+\infty} \left[P\left(\frac v t \right) \sin\, v +Q\left(\frac v t \right) \cos\,v \right]dv, \quad \mbox{where}\ Q(u):= \frac{-uq(u)}{p^2(u)+q^2(u)}.
$$
This shows that all derivatives $v^{(j)}(t,\xi), \ t >0,$ exist. 
\hfill $\square$
\end{proof}
It follows from (\ref{ev}) that $v(\pm\infty,\xi) =0$ as the Fourier transform of a function from $L_1(\R)$. In fact,  some additional work shows that 
actually $v(t)$ is exponentially decaying at $\pm \infty$:
\begin{lemma}  \label{Le2} If 
 equation (\ref{leq}) is hyperbolic then $v\in W^{1,1}(\R)$. In addition, $|v(t)|\leq Ce^{-\gamma |t|},$ $t\in \R,$ for some positive $C, \gamma$ and $v'' \in L_1(\R_\pm)$ (so that $v'(\pm\infty,\xi)=0$).
\end{lemma}
\begin{proof}  A simple inspection of the characteristic equation 
 \begin{equation}\label{eK}
z^2-cz-d=\xi e^{-chz}\int_\R e^{-zs}K(s)ds,
\end{equation}
shows that, in view of the hyperbolicity of  (\ref{leq}),  there exists $\gamma >0$ such that the vertical strip $\{|\Re z| < 2\gamma\}$ does not contain roots of (\ref{eK}). But then we can shift the path of integration in the inversion formula for the Laplace transform (e.g. see \cite[p. 88]{AGT}) to obtain
$$ 
v(t,\xi) = \frac{e^{\pm\gamma t}}{2\pi }\int_{-\infty}^{+\infty} \frac{e^{iut}du}{(\pm \gamma+iu)^2-c(\pm \gamma +i u)-d -\xi \int_\R K(s)e^{-(\pm\gamma+iu)s}ds} = {e^{\pm\gamma t}}\sigma_\pm(t), 
$$
where $\sigma_\pm(\infty)=0$. Next, by (\ref{V1}), 
$$
v'(t,\xi) = -\frac{1}{2\pi }\lim_{T\to +\infty}\int_{-T}^{T} \frac{iue^{iut}du}{u^2+ci u+d +\xi e^{-iu ch}\int_\R K(s)e^{-ius}ds}, \ \  t\not=0.
$$
Therefore similarly, for $t\not=0$, 
$$
v'(t,\xi) = \frac{e^{\pm\gamma t}}{2\pi }\lim_{T\to +\infty}\int_{-T}^{T} \frac{(\pm\gamma + iu)e^{iut}du}{(\pm\gamma + iu)^2-c(\pm\gamma + iu)+d - \xi \int_\R K(s)e^{-(\pm\gamma + iu)s}ds},
$$ 
where the latter limit also exists in $L^2(\R)$ and represent the Fourier transform  of an element of $L_2(\R)$. 
Thus $v'(t,\xi) = e^{\pm\gamma t}\rho_{\pm}(t), \ t \not =0,$ where  $\rho_\pm \in L_2(\R)$. By the H\"older inequality, $v' \in L_1(\R_\pm)$ so that $v' \in L_1(\R)$. Finally, since $v$ satisfies the equation (\ref{leq}) on $\R_\pm$, we conclude that $v''(t)=cv'(t) +v(t) + \xi K*v(t-ch)$ also belongs to  $L_1(\R_\pm)$.
\hfill $\square$
\end{proof}
\begin{lemma}  \label{Le3} If $v\in W^{1,1}(\R)$ and function $f$ is continuous and bounded then 
$v*f\in C_b^1(\R)$  and $(v*f)' = v'*f$.  \end{lemma}
\begin{proof} Clearly, $|v*f(t)|\leq \sup_{t\in \R}|f(t)||v|_1, \ t \in \R,$ and 
$$
v*f(t+\delta)- v*f(t)= \int_\R f(t-s)(v(s+\delta)- v(s))ds = \int_\R f(t-s)\int_s^{s+\delta}v'(u)duds=  
$$
$$
\delta \int_\R v'(u) du \frac 1 \delta  \int_{u-\delta}^{u}f(t-s)ds \ \  \mbox{because of }\ 
\int_\R \int_s^{s+\delta}|v'(u)|duds =\delta|v|_1 <\infty.$$
 $$
\hspace{-17mm} \mbox{Since \ }\ \left| \frac 1 \delta  \int_{u-\delta}^{u}f(t-s)ds\right|\leq  \sup_{t\in \R}|f(t)|, \quad \lim_{\delta \to 0} \frac 1 \delta  \int_{u-\delta}^{u}f(t-s)ds= f(t-u), 
 $$
we conclude that $v*f$ is differentiable on $\R$ and $(v*f)' = v'*f$.  Note also that $|v'*f(t)|\leq \sup_{t\in \R}|f(t)||v'|_1, \ t \in \R.$
\hfill $\square$
\end{proof}
Our next goal is, given a bounded continuous function $f: \R \to \R$ (i.e. $f\in C_b(\R)$),   to prove 
the uniqueness of  bounded solution $y(t)$ for the hyperbolic equation (\ref{leq}) and 
to justify the representation $y(t) = -v*f(t)$:
\begin{corollary}\label{com}  Suppose that $f\in C_b(\R)$.  If 
 equation (\ref{leq}) is hyperbolic  and $v$ is the associated fundamental solution then the formula  $u=-v*f$ gives the unique  $C^2$-smooth bounded solution of  (\ref{leq}). 
\end{corollary}
\begin{proof} Invoking Lemmas \ref{Le1}, \ref{Le2}, \ref{Le3} and using the formula
$$
u(t) = -\int_{-\infty}^tv(t-s,\xi)f(s)ds -\int^{+\infty}_tv(t-s,\xi)f(s)ds, 
$$
we find easily that $$u'(t) = -\int_{\R}v'(t-s,\xi)f(s)ds, \ u''(t) = -\int_{\R}v''(t-s,\xi)f(s)ds -f(t)(v'(0^+,\xi)-v'(0^-,\xi))$$
are continuous, bounded and satisfy equation (\ref{leq}). 

On the other hand, assume that $u(t)$ is some classical bounded solution of  equation (\ref{leq}). Then it is 
easy to see from (\ref{leq})  that $u'(t), u''(t)$ are also bounded on $\R$ and 
$$
u''* v(t) = \int_\R u''(s)v(t-s)ds = \int_\R v'(t-s)u'(s)ds = 
u(t) + \int_\R v''(t-s)u(s)ds, $$ $$u'* v(t) = \int_\R u'(s)v(t-s)ds =  \int_\R u(s)v'(t-s)ds. 
$$
In consequence, considering the convolution of equation  (\ref{leq}) with the fundamental solution, 
we find that $u(t) + v*f(t)=0$. 
\hfill $\square$
\end{proof}
\subsection{Continuous function $v(t,\xi)$  as a distributional solution of a non-local  equation}\label{Sec22}\noindent It is worthwhile to analyse the fundamental solution $v(t,\xi)$ and some of its properties from the point of view of the theory of distributions. 
The distributions will be regarded in the standard way, as elements  of the dual space $\mathcal{D}'(\R)$ \cite{Str} (recall that  the space $\mathcal{D}(\R)$ of test functions consists of compactly supported smooth  functions). 
 This perspective is quite useful  since it helps to cope with more general non-local delayed differential  operators
$$
\mathcal{L} \phi(t)=\phi^{(n)}(t) +a_{n-1}\phi^{(n-1)}(t)+ \dots + a_1\phi'(t)+a_0\phi(t) +b_0\int_\R K(t-s) \phi(s)ds+ \sum_{j=1}^mb_j\phi(t-h_j). 
$$
We assume here that $\phi \in \mathcal{D}(\R)$, $n \geq 2$, $a_i,b_j, h_j \in \R$ and that the operator $\mathcal{L} $ is hyperbolic in the sense that
the characteristic function $\chi(z,\mathcal{L})$ determined from $(\mathcal{L})(e^{zt}) = \chi(z,\mathcal{L})e^{zt}$ does not have zeros on the imaginary axis. Obviously,  this form of $\mathcal{L}$ includes   the particular 
case of the operator defined by  the left-hand side of equation (\ref{leq}).

Consider the  formally adjoint operator $\mathcal{L}^* $ defined by 
$$
\mathcal{L}^* \psi(t)=(-1)^n\psi^{(n)}(t) +(-1)^{n-1}a_{n-1}\psi^{(n-1)}(t)+ \dots - a_1\psi'(t)+a_0\psi(t) +
$$
$$b_0\int_\R K(s-t) \psi(s)ds+ \sum_{j=1}^mb_j\psi(t+h_j), \quad \psi \in \mathcal{D}(\R). $$
Clearly, for all $\phi,\psi \in \mathcal{D}(\R)$, it holds that $ \mathcal{L} \phi,  \mathcal{L}^* \psi \in L^1(\R)$
and 
$$
\int_\R \psi (t) \mathcal{L} \phi(t)dt = \int_\R \phi (t) \mathcal{L}^* \psi(t)dt. 
$$
We have the following
\begin{lemma}  \label{ds} Suppose that 
$\mathcal{L}$  is hyperbolic. Then  function 
\begin{equation}\label{evG}
v(t) = \frac{1}{2\pi }\int_{-\infty}^{+\infty} \frac{e^{iut}du}{\chi(iu,\mathcal{L})}, \quad  t \in \R, 
\end{equation}
is continuous, bounded and Lebesgue integrable on $\R$. Moreover, it is  a distributional solution of the equation $\mathcal{L} v(t) = \delta(t)$, 
where  $\delta(t)$  is the Dirac delta function:  
$$
\int_\R v(t) \mathcal{L}^* \phi(t)dt = \phi(0) \quad \mbox{ for all}\  \phi  \in \mathcal{D}(\R).
$$
In consequence, $v(t)$ is $C^{n}$-smooth on $\R\setminus\{0\}$ and $\mathcal{L} v(t) = 0$ for all 
$t\not=0$.
\end{lemma} 
\begin{proof} From $1/\chi \in L^1(\R)$ we infer that $v \in C(\R)$,  $v(\pm\infty) =0$. In fact, repeating the argument given in the first lines of the proof of Lemma \ref{Le2}, we obtain that  $|v(t)|\leq Ce^{-\gamma |t|},$ $t\in \R,$ for some positive $C, \gamma$.  
Now, using the Fubini theorem and integrating by parts, we find that 
$$
\int_\R v(t) \mathcal{L}^* \phi(t)dt = \frac{1}{2\pi }\int_\R \mathcal{L}^* \phi(t)dt \int_{\R} \frac{e^{iut}du}{\chi(iu,\mathcal{L})}=
\frac{1}{2\pi }\int_{\R} \frac{du}{\chi(iu,\mathcal{L})}\int_\R e^{iut} \mathcal{L}^* \phi(t)dt =
$$
$$
\frac{1}{2\pi }\int_{\R} \frac{du}{\chi(iu,\mathcal{L})}\int_\R\phi(t) \mathcal{L} e^{iut} dt = \frac{1}{2\pi }\int_{\R} \frac{du}{\chi(iu,\mathcal{L})}\int_\R\phi(t)\chi(iu,\mathcal{L})e^{iut} dt =$$
$$
\frac{1}{2\pi }\int_{\R} du\int_\R\phi(t) e^{iut} dt = \phi(0). 
$$
In the last line, we are using the inversion formula for the Fourier transform \cite{VAZ}. 

Finally,  we observe that, on each interval $(\alpha, \beta)$ disjoint with 
$\{0\}$, continuous $v(t)$ is a distributional solution of the following inhomogeneous linear ordinary equation with constant coefficients and continuous right-hand side 
$F(t)$: 
$$
v^{(n)}(t) +a_{n-1}v^{(n-1)}(t)+ \dots + a_1v'(t)+a_0v(t) = F(t), \quad t \in (\alpha, \beta), $$
$$F(t):= - b_0\int_\R K(t-s) v(s)ds-  \sum_{j=1}^mb_jv(t-h_j).
$$
It is well known \cite{Shilov} that, in such a case, $v(t)$ is also a classical solution on $(\alpha, \beta)$ of the latter equation. 
\hfill $\square$
\end{proof}
Now, since for $f \in C_b(\R)$ and continuous $v \in L^1(\R)$
$$
-\int_\R v*f (t) \mathcal{L}^* \phi(t)dt = \int_\R f(s) ds \int_\R v(t-s)\mathcal{L}^* \phi(t+s)dt= \int_\R f (s)\phi(s)ds,
$$
we obtain the following
\begin{corollary} For each fixed $f \in C_b(\R)$, the function  $u=-v*f$ is a distributional solution of the equation 
$\mathcal{L} u(t)+f(t)=0$. In consequence  \cite{Shilov}, $u(t)$ is also a bounded $C^{n}$-smooth classical solution on $\R$ of this equation. 
\end{corollary}
\subsection{A criterion of negativity of the fundamental solution $v(t,\xi)$}\noindent In this subsection, assuming the hyperbolicity of equation (\ref{leq}), we establish a criterion of negativity of  its fundamental solution $v(t,\xi)$.  It is well known (and it is straightforward to check) that $v(t,\xi)<0, \ t \in \R,$ in the local and non-delayed case when $\xi=0, d >0$  and 
\begin{equation}\label{wn}
v(s,0)=\min\{e^{\lambda_1(0) s}, e^{\lambda_0(0) s} \}/(\lambda_1(0)-\lambda_0(0)),
\end{equation}
with $\lambda_1(0) <0 <\lambda_0(0)$ being the roots of the characteristic equation $\lambda^2-c\lambda -d=0$.
Suppose now that $K(s)$ satisfies  {\rm \bf(K)} and $\xi \geq 0, \ d +\xi >0 $. Then it is easy to see that there exists a unique positive number $\xi^*$ such that $\chi(z,\xi)$ has both positive and negative finite zeros if and only if $\xi \in [0,\xi^*]$.   In fact, since $\chi^{(4)}(s,\xi) < 0$ for all $s \in \R$,  function $\chi(s,\xi), \ s \in \R$, for $\xi \in (0,\xi^*]$ can have at most four real zeros $\lambda_j(\xi)$, all of them being finite if 
$$
\int_{-\infty}^0K(s)ds\int^{+\infty}_0K(s)ds\not=0. 
$$
In such a case, we will order them as $\lambda_2(\xi) \leq \lambda_1(\xi) <0< \lambda_0(\xi) \leq \lambda_{-1}(\xi)$.  
If $\int_{-\infty}^0K(s)ds=0$ and  $\xi \in [0,\xi^*]$, then there are exactly two negative and one positive finite roots 
$\lambda_2(\xi) \leq \lambda_1(\xi) <0< \lambda_0(\xi)$; by definition, we set $\lambda_{-1}(\xi)=+\infty$. A similar situation occurs when $\int^{+\infty}_0K(s)ds=0$, $\xi \in [0,\xi^*]$, where it is convenient to set $\lambda_{2}(\xi)=-\infty$. Finally, we set $\lambda_{2}(0)=-\infty, \ \lambda_{-1}(0)=+\infty$.
Observe that in either case the biggest negative root $\lambda_1(\xi)$ and the smallest positive root 
$\lambda_0(\xi)$ are finite numbers. 
\begin{lemma} \label{L12} Suppose that $\xi \in [0,\xi^*]$, $c^2+d>0$. Then in the closed strip
$$\overline{\Pi}(\lambda_2,\lambda_{-1}):=\{z: \lambda_2(\xi) \leq \Re z \leq \lambda_{-1}(\xi) \}$$ the function 
$\chi(z,\xi)$ does not have zeros different from $\lambda_j(\xi), \ j =-1,0,1,2$. 
\end{lemma}
\begin{proof} First, note that only zero of $\chi(z,\xi)$ on the line $\Re z = \lambda_j(\xi),$ where $j \in \{-1,0,1,2\}$, is $\lambda_j(\xi)$. 
Indeed, if $z_j, \ \Re z_j = \lambda_j(\xi),$ denotes another root of equation (\ref{eK})
then, using the factorization $z^2-cz-d= (z-A)(z-B)$ with real $A, B$, we get the following contradiction:
$$|\lambda_j^2-c\lambda_j-d|< |z_j-A||z_j-B|=|z_j^2-cz_j-d| \leq \xi e^{-ch\lambda_j}\int_\R e^{-\lambda_j s}K(s)ds= |\lambda_j^2-c\lambda_j-d|.$$
Next, the right hand side of (\ref{eK}) is uniformly bounded in each closed strip $\overline{\Pi}(a,b)$ by 
$$
\xi^* \int_\R (e^{-a(ch+s)}+e^{-b(ch+s)})K(s)ds, 
$$
while $z^2-cz-1 \to \infty$ as $z \to \infty$.
In consequence, there exists positive $C$ which does  not depend on $\xi$ such that 
each zero $z_k$ of $\chi(z, \xi),\ \xi \in [0,\xi^*]$, in $\overline{\Pi}(\lambda_1(\xi),\lambda_0(\xi))$ satisfies 
$|\Im z_k|\leq C$.  Since $\lambda_j(\xi), j =0,1,$ are continuous functions of $\xi$ and $\overline{\Pi}(\lambda_1(0),\lambda_0(0))$ does not contain non-real roots of $\chi(z,0)$, we find that either $\overline{\Pi}(\lambda_1(\xi),\lambda_0(\xi))$ contains only two zeros of $\chi(z,\xi)$ for all $\xi \in [0,\xi^*]$ or there exists $\xi_1 \in (0,\xi^*]$
and complex zero $z_1$ of $\chi(z,\xi)$ such that $\Re z_1\in \{\lambda_1(\xi_1),  \lambda_0(\xi_1)\}$. However,  as we have just proved, the latter  can not happen.  

Finally, suppose that  $\chi(z_0,\xi) =0, \ \lambda_2(\xi)<x_0:= \Re z_0 < \lambda_1(\xi)$,
(the case when $\lambda_0(\xi)<\Re z_0 < \lambda_{-1}(\xi)$ can be treated  analogously). Then we get the following contradiction 
$$
|z_0^2-cz_0-d| \leq \xi e^{-ch x_0}\int_\R e^{-x_0 s}K(s)ds < |x_0^2-cx_0-d| \leq |z_0-A||z_0-B| = |z_0^2-cz_0-d|.
$$
This completes the proof of the lemma.  
\hfill $\square$
\end{proof}
\begin{lemma} \label{Le6} Suppose that $\xi \in [0,\xi^*],\ d>0, \ h \geq 0$. Then $v(t,\xi)<0$ for all $t \in \R$, $\xi \in [0,\xi^*]$.  Moreover, $v(t,\xi)$ is sign-changing  on $\R$ for each $\xi>\xi^*$ close to $\xi^*$. 
\end{lemma}
\begin{proof}  Due to Lemma \ref{L12}, equation (\ref{leq}) is hyperbolic and therefore the fundamental solution exists. The proof of its negativity  is divided in several steps.
Recall that if  $\xi \in [0,\xi^*)$   then 
$\chi'(\lambda_0(\xi)) >0$, $\chi'(\lambda_1(\xi))<0$ and  $\lambda_j(\xi),\  j =0,1,$ are simple zeros of $\chi(z,\xi)$.

\underline{Claim I}. For each non-negative  $\xi_0 < \xi^*$ there exist real numbers $\nu_0, \nu_1$, 
a 
neighbourhood $\mathcal{O}\ni \xi_0$ and positive constants $K, L$ such that, for all  $ \xi \in \mathcal{O}$, 
\begin{eqnarray}\label{fico}
\lambda_2(\xi) +L < \nu_1 < \lambda_1(\xi)-L, \ \lambda_0(\xi) +L < \nu_0< \lambda_{-1}(\xi) -L, \\
\label{ficoR}
v(t,\xi) = \rho_j(\xi)e^{\lambda_j(\xi)t} + r_j(t,\xi), \quad |r_j(t,\xi)| \leq Ke^{\nu_j t}, \quad  (-1)^{j+1}t \geq 0, \ j=0,1, \end{eqnarray}
where $\rho_j(\xi) = (-1)^{j+1}/\chi'(\lambda_j(\xi))<0, \ j=0,1,$ depend continuously on $\xi$.  

We prove this claim for $j=1$, the other case being similar. 
Fix some $\xi_0 <\xi^*$ and $\nu_1 \in (\lambda_2(\xi_0),\lambda_1(\xi_0))$. Then  we can choose a neighbourhood
$\mathcal{O} \ni \xi_0$ and $L>0$ sufficiently small to meet the condition (\ref{fico}) for all 
$\xi \in \mathcal{O}$. 
 Next, after moving the integration path in the inversion formula (\ref{ev}) from $\Re z = 0$ to $\Re z = \nu_1$, we obtain that, for $t \geq 0$, $ v(\xi, t) = $
 $$
 \frac{e^{\lambda_1t}}{\chi'(\lambda_1,\xi)} + \frac{1}{2\pi i}\int_{\nu_1-i\cdot\infty}^{\nu_1+i\cdot\infty}\frac{e^{tz}dz}{\chi(z,\xi)}= -\frac{e^{\lambda_1t}}{|\chi'(\lambda_1,\xi)|} +\frac{e^{\nu_1 t}}{2\pi}\int_{-\infty}^{+\infty}\frac{e^{ist}ds}{\chi(\nu_1+is,\xi)} =:e_1(t) + e^{\nu_1 t} q(t),
 $$
 where $q(\pm\infty)=0$ and  
 $$
 |q(t)| \leq K= \sup_{\xi \in \mathcal{O}}\frac{1}{2\pi} \int_{-\infty}^{+\infty}\frac{ds}{|\chi(\nu_1+is,\xi)|}. $$

Claim I implies that  exponentially decaying function $v(t,\xi)$, $\xi \in \mathcal{O}$,  is negative at $\pm\infty$. In particular,  there exists a leftmost point $T_+(\xi)\geq 0$ such that $v(t,\xi) <0$ for all $t > T_+(\xi)$. Analogously, $T_-(\xi)\leq 0$ denotes the rightmost point such that $v(t,\xi) <0$ for all $t <T_-(\xi)$. By (\ref{wn}), 
$T_\pm(0)=0$.

\vspace{2mm}

\underline{Claim II}. $v(t,\xi)$ is a continuous function of $t \in \R,\ \xi \in [0, \xi^*]$.  Furthermore, 
 $v(t,\xi)$ is bounded on $\R$, uniformly with respect to $\xi \in [0, \xi^*]$. 

Indeed,  observe that  the function $g(u,\xi):= 1/\chi(iu,\xi)$ is continuous on $\R\times [0,\xi^*]$ and
 $|g(u,\xi)|\leq G(u),$ $u \in \R$, $\xi \in [0,\xi^*]$, 
with $G \in L_1(\R)$ defined by 
$$G(u)=\left\{\begin{array}{cc} \max\{1/|\chi(is,\xi)|,\ |s|\leq 2+\xi^*, \xi \in [0,\xi^*]\}, & |u|\leq 2+\xi^*, \\ 
1/(u^2 -\xi^*), &  |u|> 2+\xi^*, \end{array}\right.
$$
(recall that $\chi(\lambda,\xi) \not= 0$ when $\Re \lambda =0,$ $\xi \in [0, \xi^*]$). In particular,  to prove the continuity statement,  it suffices to apply the Lebesgue dominated convergence theorem to (\ref{ev}).

\vspace{2mm}

\underline{Claim III}. On each closed  interval $[0,\zeta] \subset [0, \xi^*)$,  functions $T_\pm(\xi)$ are bounded.  

Due to the compactness of  $[0,\zeta]$, it is enough to prove that $T_\pm(\xi)$ are locally bounded.  
For example, consider $T_+(\xi)$ for $\xi \in \mathcal{O}$. We have that either $T_+(\xi) =0$ or $T_+(\xi) >0$ and 
$$
0= v(T_+(\xi)) \leq \rho_1(\xi)e^{\lambda_1(\xi)T_+(\xi)} +  Ke^{\nu_1 T_+(\xi)}, \ \xi \in \mathcal{O}.
$$
In the latter case, for all $\xi \in \mathcal{O}$, 
$$
T_+(\xi) \leq \frac{1}{\lambda_1(\xi) - \nu_1}\ln \frac{K}{|\rho_1(\xi)|} \leq \frac{1}{L}\ln (K\sup_{\xi \in \mathcal{O}} |\chi'(\lambda_1(\xi),\xi)|). 
$$

\underline{Claim IV}.   $v(t,\xi)<0$ for all $t \in \R$, $\xi \in [0,\xi^*]$.   Furthermore, $v(t,\xi)$ is sign-changing  on $\R$ for $\xi>\xi^*$ close to $\xi^*$. 

Let $\xi_c \in [0,\xi^*]$ be the maximal number such that 
$v(t,\xi) < 0,\ t \in \R,$ for all $\xi \in [0,\xi_c)$.  For a moment, suppose that $\xi_c < \xi^*$. Since $v(t,0) <0$ for all $t\in \R$, due to Claims II and III, $\xi_c>0$ and $v(t_c,\xi_c)=0$ for some 
$T_-(\xi_c) \leq t_c \leq T_+(\xi_c)$.  Since $v(\pm \infty, \xi_c)=0$, there exists two numbers $a_c< b_c$ 
where $v(t,\xi_c)$ reaches its absolute minima on the half-lines $(-\infty, t_c]$ and $[t_c, +\infty)$, respectively. 
Clearly, either $a_c$ or $b_c$ is different from zero. For instance, suppose that $a_c \not=0$. Then $v(t,\xi_c)$ 
is differentiable at this point where $v(a_c,\xi_c) <0, v'(a_c,\xi_c)=0, v''(a_c,\xi_c)\geq 0$ and $K*v(a_c-ch) \leq 0$. 
Obviously, this contradicts equation (\ref{leq}) with $f(t)\equiv 0$ at $t=a_c\not=0$. This shows that $\xi_c=\xi^*$ and also implies that $v(t,\xi^*)\leq 0$ for all $t\in \R$. 
Now, $\xi=\xi^*$ is a bifurcation point for some real zero of  $\chi(z,\xi)$: for instance, 
suppose that for $\lambda_2(\xi^*)=\lambda_{1}(\xi^*)<0$.  Clearly, $\chi'(\lambda_1(\xi^*),\xi^*) =0$ while $\chi''(\lambda_1(\xi^*),\xi^*) \not=0$ for otherwise  $\lambda_1(\xi^*)$ would be a triple negative zero of $\chi(z,\xi^*)$. 
 Then, using the inversion formula again, we find that  $v(t,\xi^*)$ is negative at $+\infty$ because of 
the relation $\lim_{t \to +\infty} v(t,\xi^*) t^{-1}e^{-\lambda_1(\xi^*)t}=1/\chi''(\lambda_1(\xi^*),\xi^*) <0.$ Thus the above proof of negativity also works for $v(t,\xi^*)$.
Next, for all $\xi>\xi^*$ 
close to $\xi^*$,  function $\chi(z,\xi)$ has two simple complex conjugated zeros  $\lambda_{1\pm}(\xi):= p(\xi)\pm i q(\xi), \ \lambda_{1\pm}(\xi^*):=\lambda_{1}(\xi^*),$ such that the strip ${\Pi}(p(\xi),0]$ does not contain any zero of $\chi(z,\xi)$.  Now, assuming that 
$v(t,\xi) \geq 0, \ t \in \R,$ we infer from \cite[Theorem 5b, p. 58]{WID} that the singularity  of the Laplace transform $1/\chi(z,\xi)$ of $v(t,\xi)$, which is rightmost on the half-plane $\Re z \leq 0$, should be a real number and not complex as $\lambda_{1\pm}(\xi)$. This contradiction proves the second part of Claim IV. \hfill $\square$
\end{proof}
\begin{remark} It can be proved in Lemma \ref{Le6} that actually $v(t,\xi)$ is oscillating (either at $+\infty$ or $-\infty$) for $\xi>\xi^*$ close to $\xi^*$. Observe that the assumption of smallness of $\xi-\xi^*$ is used only in order to assure the existence of  zeros of $\chi(z,\xi)$ in the both half-planes $\{\Re z >0\}$ and $\{\Re z <0\}$.  If $K$ has a compact support (for example, if $K$ is the delta of Dirac),  this requirement is fulfilled automatically. However, 
we do not know whether the smallness condition can be omitted for general kernels $K$.  This is because  there is some lack of a priori knowledge regarding the distribution of zeros of $\chi(z,\xi)$: nevertheless,  since the entire function $\chi(iz,\xi)$ is of class A and is of completely regular growth, some general  information about this distribution can be found in \cite[Chapter V, Theorem 11]{Levin}. 
\end{remark}
\begin{remark} The positivity of the fundamental solution (or of the Green function) for solving initial/boundary value problems for delayed differential equations is an important  topic 
of  the theory of functional differential equations. See the recent monographs \cite{ABBD,Gil} for more references concerning this problem. 
\end{remark}

The final result of this section shows that the geometric form of $v(t,\xi)$ for $\xi \in [0,\xi^*]$ is quite similar to the shape of $v(t,0)$ given in (\ref{wn}):
\begin{corollary} If $\xi \in [0,\xi^*]$ then $v(t,\xi)$ has a unique minimum point at $t=0$. Moreover, $v(t,\xi)$ is strictly monotone on $\R_-$ and  $\R_+$. It is also strictly convex on $\R_-$. 
\end{corollary}
\begin{proof} Indeed, as we have seen in the proof of Lemma \ref{Le6}, $v'(t,\xi)$ can not change the 
sign on $(-\infty, 0)$ and $(0, +\infty)$ because otherwise $v(t,\xi)$ reaches a local minimum at some 
point of $\R\setminus\{0\}$.  By the same reason, $v'(t,\xi)$ can not vanish on an open interval. Finally, 
observe that all this implies that $v''(t,\xi) >0$ for $t <0$. 
\hfill $\square$
\end{proof}
\section{Proof of Theorem \ref{mai1}}\label{S3}
\paragraph{Case I: $(h,c)\in \mathcal{D}_{\frak L}$}
We are assuming  that the characteristic equation $\chi_\kappa(z)=0$ has at least one negative and one 
positive simple roots $\lambda_1(|g'(\kappa)|)<0 < \lambda_0(|g'(\kappa)|)$. Therefore, for 
sufficiently small $\delta >0$ we have that $\xi= |g'(\kappa)|+\delta\leq \xi^*$ and  the equation $\chi(z,\xi)=0$ also has 
at least one negative and one 
positive simple roots $\lambda_1(\xi), \lambda_0(\xi)$:
$$\lambda_1(\xi)< \lambda_1(|g'(\kappa)|)<0 < \mu_0 <\mu_1 < \lambda_0(|g'(\kappa)|)< \lambda_0(\xi).$$
With $g_1(s)=(g(s) + \xi s)/(1+\xi)$,   the profile equation (\ref{yp}) can be rewritten as 
\begin{equation}\label{ypa}
y''(t) -cy'(t)-y(t) - \xi\int_\R K(t-s) y(s-ch)ds + (1+\xi)\int_\R K(t-s) g_1(y(s-ch))=0. 
\end{equation}
By Corollary \ref{com}, this equation has at least one bounded solution $\phi(t)$ if and only 
if 
\begin{equation} \label{cn}
\phi(t) = \mathcal{N}\phi(t), \ \mbox{where} \  \mathcal{N}\phi(t):= \int_\R N(t-s) g_1(\phi(s-ch))ds, \  N(s)= -(1+\xi)v*K(s).
\end{equation}
In virtue of Lemma \ref{Le6}, the following properties of $N(s)$ are immediate: $N(s) >0,  \ s \in \R,$  $\int_R N(s)ds=  1,$ and 
 \begin{equation}\label{sc}
\int_R e^{-zs}N(s)ds = -(1+\xi)\frac{\int_\R e^{-zs}K(s)ds}{\chi(z,\xi)} <\infty  \ \mbox{for all}\  z \in (\lambda_1(\xi),  \lambda_0(\xi)). 
\end{equation}
On the other hand, $g_1(s)$ is strictly increasing on $[0,\kappa]$ where $g_1(\kappa)=\kappa, \ g_1(0)=0 $ and 
$$
\ g_1'(\kappa)=\frac{\delta}{|g'(\kappa)|+\delta}\in (0,1),\ g_1(s)= \frac{g(s)+\xi s}{1+\xi} \leq g_1'(0)s = \frac{g'(0)s+\xi s}{1+\xi}. 
$$
Therefore nonlinear convolution equation  (\ref{cn})  can be analysed within the framework of theory developed in 
\cite{SEDY}. Particularly, Theorem 7 in \cite{SEDY} guarantees the existence of a positive solution $y=\phi(t)$ to 
(\ref{ypa}) satisfying the conditions $\phi(-\infty)=0, \ \phi(+\infty)=\kappa$.  Moreover, it is easy to see that solution 
$\phi(t)$ provided by  \cite[Theorem 7]{SEDY} is a non-decreasing one if $g_1(s)$ is a non-decreasing function. For the sake of completeness,  in Remark \ref{RR8} below, we indicate  the corresponding change in the proof of 
 \cite[Theorem 7]{SEDY}.  Now,  due to  the positivity of $N(s)$, the profile   $\phi(t)$  is actually a strictly increasing function:  if $t_2>t_1$ then $\phi(t_2-s) \geq \phi(t_1-s), \ s \in \R, \  \phi(t_2-s) \not\equiv \phi(t_1-s)$, so that   
$$
\phi(t_2) = \int_\R N(s-ch) g_1(\phi(t_2-s))ds >   \int_\R N(s-ch) g_1(\phi(t_1-s))ds =\phi(t_1). 
$$
 Hence, the proof of Case I is completed if $g_1(s)$ is increasing on $\R_+$.  Otherwise, consider some  increasing continuous and bounded function $g_2(s)$ coinciding with $g_1(s)$ on $[0,\kappa]$ and such that $g_2'(\kappa)=g_1'(\kappa)$. But then, due to the first 
part of the proof,  convolution equation (\ref{cn}) where $g_1$ is replaced with $g_2$ has a monotone solution $\phi:\R \to [0,\kappa]$.  Since $g_1(s)\equiv g_2(s)$ on $[0,\kappa]$, the same function  $\phi(t)$ solves (\ref{cn}). 

\paragraph{Case II: $(h,c)$ belongs to the boundary of the set  $\mathcal{D}_{\frak L}$} In such a case, there exists 
a sequence $\{(h_j,c_j)\}$ of points in $\mathcal{D}_{\frak L}$ converging to $(h,c)$.  From Case I we conclude that 
for each point $\{(h_j,c_j)\}$ there exists a monotone positive solution  $y= \phi_j(t),$  $\phi_j(-\infty)=0,$  $\phi_j(-\infty)=\kappa$,  satisfying the profile equation 
$$
y''(t) -c_jy'(t)-y(t) + \int_\R K(t-s) g(y(s-c_jh_j))=0. 
$$
Since this equation is translation invariant, we can assume that $\phi_j(0)=\kappa/2$ for each $j$. Then it follows that $\phi_j(s)$ has a subsequence $\phi_{j_k}(t)$ converging (uniformly on compact subsets of $\R$) to a positive monotone solution 
$\phi(t), \ \phi(0) =\kappa/2,$ of the limit equation (\ref{yp}) (e.g. see \cite{SEDY} or \cite[Section 6]{TAT} for more details.  Now,  the monotonicity of $\phi(t)$ implies that the boundary conditions in (\ref{yp}) are also satisfied (e.g. see Remark \ref{RR8} below).  This completes the proof of Theorem \ref{mai1}. \hfill $\square$
\begin{remark}\label{RR8} In  order to solve the following slightly modified version 
 \begin{equation}\label{meq}
\phi := \mathcal{A}\phi, \quad  \mbox{where} \  \mathcal{A}\phi(t):= \int_\R N(s) \gamma_{n}(\phi(t-s-ch))ds, 
\end{equation}  
 $$\gamma_n(s):=\left\{\begin{array}{ll}g'_1(0)s,& \textrm{ for } s\in[0,1/n],\\
 \max\{g'_1(0)/n, g_1(s)\},& \textrm{ when } s \geq 1/n, \end{array} \right.$$
 of equation $\mathcal{N}\phi =\phi$ in Section 4 of \cite{SEDY}, 
we can use the iteration procedure $\phi_{j+1}= \mathcal{A}\phi_{j}, \ j =0,1,\dots, \ \phi_0(s)=n^{-1}\exp({\mu_0s)}, s \in \R,$ instead of the Schauder fixed point theorem. For a small positive $\epsilon >0$, set  $\phi^-(t)=\phi_0(t)(1-e^{\epsilon t})\chi_{\R_{-}}(t)$, where $\chi_{\R_{-}}(t)$ is the characteristic function of $\R_-$.  Since $\gamma_n(s), \phi_0(t)$ are non-decreasing functions and  $\phi_-(t) \leq \mathcal{A}\phi_{-}(t)\leq \phi_1(t) \leq \mathcal{A}\phi_{0}(t) \leq \phi_0(t)$, we conclude that each $\phi_j(t), \ j \in \N,$ is also a non-decreasing function and  $\phi_0(t) \geq \phi_2(t)\geq \cdots \geq \phi_j(t) \geq \dots\geq \phi_-(t) $.  Then the limit $\phi(t)= \lim_{j\to+\infty}\phi_j(t)$ should be  a positive non-decreasing and bounded solution of the equation $\phi = \mathcal{A}\phi$. Taking the limit 
in (\ref{meq}) as $t\to \pm\infty$, we obtain that $\phi(\pm\infty) = \gamma_n(\phi(\pm\infty))$ that immediately implies that   $\phi(-\infty)=0, \ \phi(+\infty)=\kappa$. 
\end{remark}

\section{Proof of Theorem \ref{main2}} \label{S4}
\noindent In this section, we show how the use of convolution equation (\ref{ypa}) helps to extend the front uniqueness result  established for  equation (\ref{17b}) with monotone birth function $g$   (e.g. see \cite[Theorem 1.2]{TPT}) on the case of non-local and non-monotone model (\ref{17a}). 
\begin{lemma} \label{l6} Fix some $(c,h)\in  \mathcal{D}_{\frak L}$ and suppose that $\phi, \psi:\R\to (0,\kappa]$ are two wavefront profiles satisfying equation (\ref{yp}) and such that $\phi$ is monotone and,  for some finite $T$, 
\begin{equation}\label{nado}
\phi(t) < \psi(t),  \quad t <  T. 
\end{equation}
Then $\phi(t) < \psi(t)$  for all $t \in \R$. 
\end{lemma}
\begin{proof} 
Set $a_* = \inf\mathcal{A}$ where
$$
\mathcal{A}: = \{a \geq 0: \psi(t)+a \geq \phi(t), \ t \in \R \}. 
$$
Note  that $\mathcal{A}\not=\emptyset$ since $[\kappa, +\infty) \subset \mathcal{A}$. Clearly, $a_* \in \mathcal{A}$.

Now, if $a_* =0$ then  $\psi(t) \geq \phi(t), \ t \in \R.$ We claim that, in fact, $\psi(t) > \phi(t), \ t \in \R.$
Indeed, otherwise we can suppose that $T$ is such that $\phi(T)=\psi(T)$.  In this way, the difference
$\psi(t)- \phi(t) \geq 0$ reaches its minimal value $0$ at $T$. Then, recalling that  $N(t)= -(1+\xi)v*K(t) >0$ for all $t\in \R$,  
we get  a contradiction: 
\begin{eqnarray}\label{pt}
0= \psi(T)-\phi(T) = \int_\R N(s-ch)(g_1(\psi(T-s))-g_1(\phi(T-s)))ds  >0. 
\end{eqnarray}
In this way,  Lemma \ref{l6} is proved when $a_* =0$ and, consequently, we have to consider the case  $a_* >0$. 
Let $\sigma >0$ be small enough to satisfy
$$\gamma_1:= \max_{s \in [\kappa-\sigma, \kappa+\sigma]} g_1'(s) < 1.$$
\underline{Case I.} 
First,  we take $Q >0$ such that $\kappa \int_Q^{+\infty}N(s-ch)ds<a_*(1-\gamma_1)$
and suppose  that $T$ is large enough to have  
\begin{equation}\label{blun}
\phi(t), \psi(t) \in (\kappa-\sigma, \kappa+\sigma), \ t \geq T-Q.
\end{equation} 
In such a case non-negative function 
$$w(t): = \psi(t) +a_* - \phi(t), \quad w(\pm \infty) = a_*>0, $$  
reaches its minimal value $0$  at some leftmost point $t_m$, where 
$
\psi(t_m) - \phi(t_m) =-a_*. 
$
Thus $\psi(t_m) < \phi(t_m)$ and therefore $t_m > T$ so that 
$$\psi(t_m-s),\ \phi(t_m-s) \in (\kappa-\sigma, \kappa+\sigma), \quad s \leq Q.$$
In consequence,  setting $\theta(s) \in (\kappa-\sigma, \kappa+\sigma)$,  we obtain
\begin{eqnarray}\label{pat}
\nonumber -a_*= \psi(t_m)-\phi(t_m) = \int_\R N(s-ch)(g_1(\psi(t_m-s))-g_1(\phi(t_m-s)))ds  = \\
\nonumber \left( \int_{-\infty}^Q  +  \int^{+\infty}_Q\right) N(s-ch)(g_1(\psi(t_m-s))-g_1(\phi(t_m-s)))ds> \\
-a_*(1-\gamma_1) + \int_{-\infty}^QN(s-ch)g'_1(\theta(s)) (\psi(t_m-s)- \phi(t_m-s))ds \geq \nonumber \\ -a_* +a_*\gamma_1 - \gamma_1a_*\int_{-\infty}^QN(s-ch)ds\geq -a_*, \nonumber \quad \mbox{a contradiction.}
\end{eqnarray}
\noindent \underline{Case II.} If (\ref{blun}) does not hold,  then, due to the convergence of profiles at $+\infty$ and the strict monotonicity of $\phi$,  we can find large $\tau>0$ and $T_1 > T$  such that 
$$
\psi(t+\tau) > \phi(t),\  t < T_1, \quad \phi(t),\psi(t+\tau) \in (\kappa-\sigma, \kappa+\sigma), \ t \geq T_1-Q.  
$$
Therefore, in view of the  previous arguments, we obtain that 
 \begin{equation}\label{bla}
\psi(t+\tau) > \phi(t), \quad t \in \R.   
\end{equation} 
Define now $\tau_*$ by
$$
\tau_*:= \inf \{\tau \geq 0: \ {\rm inequality} \ (\ref{bla}) \ {\rm holds}\}. 
$$
It is clear that $\psi(t+\tau_*) \geq \phi(t), \ t \in \R$. Now, using the same argument as in  (\ref{pt}), we conclude that either 
$\psi(t+\tau_*) \equiv \phi(t)$ with $\tau_*=0$ (a contradiction) or 
$\psi(t+\tau_*) > \phi(t), \ t \in \R$. 
In the latter  case, if  $\tau_*=0$, then Lemma \ref{l6} is  proved.  
Otherwise, $\tau_* >0$ and for each $\varepsilon \in (0,\tau_*)$ there exists a unique $T_\varepsilon > T$ such that 
$$
\psi(t+\tau_*-\varepsilon) > \phi(t), \ t <T_\varepsilon, \  \psi(T_\varepsilon+\tau_*-\varepsilon) = \phi(T_\varepsilon). 
$$
It is immediate to see that $\lim T_\varepsilon = +\infty$ as $\epsilon \to 0^+$.  Indeed, if $T_{\varepsilon_j} \to T'$  for some finite $T'$ and $\varepsilon_j \to 0^+$, then we get a contradiction: $\psi(T'+\tau_*) = \phi(T')$.  Therefore, if $\varepsilon$ is small, then 
$$
\psi(t+\tau_*-\varepsilon), \phi(t) \in (\kappa-\sigma, \kappa+\sigma), \ t \geq T_\varepsilon-Q, 
$$
that is  $\psi(t+\tau_*-\varepsilon) $ and $\phi(t)$ satisfy condition (\ref{blun}).  But then  we get 
$\psi(t+\tau_*-\varepsilon) > \phi(t)$ for all $t \in \R$, in contradiction to the definition of $\tau_*$.  This means that 
$\tau_* =0$ and the proof of Lemma \ref{l6} is  completed.   \hfill $\square$
\end{proof}
\begin{corollary} \label{coco}Fix some $(c,h)$ in the closure of $\mathcal{D}_{0}$ and suppose that equation 
(\ref{yp}) possesses  two monotone wavefronts $\phi$ and $\psi$. Then there exists $s_0, s_1\in \R$ and $j \in \{0,1\}$ such that either $\lim_{s\to -\infty}\phi(s+s_0)e^{-\mu_js} = 1, $ $\lim_{s\to -\infty}\psi(s+s_1)e^{-\mu_js} = 1, $ or $\lim_{s\to -\infty}\phi(s+s_0)e^{-\mu_js} s^{-1}= -1, $ $\lim_{s\to -\infty}\psi(s+s_1)e^{-\mu_js} s^{-1}= -1$.  
\end{corollary}
\begin{proof}  First, we prove that every profile satisfies one of the given asymptotic formula, with $j$ which 
might depend on the profile.  For definiteness, we will take profile $\phi$. We are going to apply some results of  \cite{AGT} to the convolution equation (\ref{cn}). 
It follows from (\ref{sc}) that the set $\{z: \sigma_K < \Re z < \gamma_K\}$, where $\sigma_K =\lambda_1(\xi), \gamma_K =  \lambda_0(\xi)$, is  the maximal open strip  of convergence for the 
Laplace transform of $N$, cf. 
\cite[Theorem 16b]{WID}.  Moreover, 
$$
\lim_{x \to \gamma_K-}\int_\R N(s)e^{-sx}ds =+\infty \ \mbox{and, in virtue of (\ref{ficoR}), } \ N(s) = O(e^{\lambda_0(\xi)s}), \ s \to -\infty.  
$$ 
Therefore, using condition (\ref{gco}) and a standard argument of the Diekmann-Kaper approach (cf. Step I of the proof of Theorem 3 in \cite{AGT}), we find that, for some $j, k \in \{0,1\}$ and $\rho>0$, the Laplace transform $\int_\R\phi(s)e^{-zs}ds$ is analytic in 
the strip $0<\Re z < \mu_j$, has a singularity at $\mu_j$, and satisfies 
$$
\frac{\chi_0(z)}{\chi(z,\xi)}\int_\R\phi(s)e^{-zs}ds = D(z), 
$$
where $D(z)$ is analytic in a bigger strip $0<\Re z < \mu_j+\rho$. Since clearly $\Phi_+(z):= \int_0^{+\infty}\phi(s)e^{-zs}ds$ is analytic in the half-plane $\{\Re z>0\}$, we conclude that the function 
$
Q(z):= D(z) \chi(z,\xi)/\chi_0(z) - \Phi_+(z)
$
is meromorphic in $0<\Re z < \mu_j+\rho$, where it has a unique singularity (a simple or double pole) at $\mu_j$. Since  
$
\Phi_-(z):= \int_{-\infty}^0e^{-sz}\phi(s)ds= Q(z)
$
for $\Re z \in (0, \mu_j)$
and $\phi(s)$ is positive and non-decreasing on $\R_-$, an application of the Ikehara theorem \cite[Proposition 2.3]{CC} yields 
the required asymptotic formula. 

Finally, we claim that $\phi$ and $\psi$ have the same asymptotic behavior at $-\infty$. 
For example, suppose that $\phi(t) \sim e^{\mu_0 t}$ and $\psi(t) \sim e^{\mu_1 t}$ as $t \to -\infty$. 
Then for every fixed $\tau \in \R$ there exists $T(\tau)$ such that $\psi(t+\tau) < \phi(t)$ for all $t < T(\tau)$.  Applying 
Lemma \ref{l6}, we obtain that  
 $\psi(s) < \phi(t)$ for every $s:=t+\tau, \ t  \in \R,$ what obviously is false. \hfill $\square$
\end {proof}
Now we are in position to finalise the proof of Theorem \ref{main2}.  By Corollary \ref{coco}, we can suppose that $\psi(t)$ and $\phi(t)$ have the same type of asymptotic behavior at $-\infty$. 
Consequently,  $\psi(t+\tau), \phi(t)$ satisfy condition (\ref{nado}) of  Lemma \ref{l6} for every small $\tau >0$. But then 
 $\psi(t+\tau) >  \phi(t)$ for every small $\tau >0$ that yields   $\psi(t) \geq  \phi(t), \ t \in \R$. By symmetry,  we also find that 
$\phi(t) \geq  \psi(t), \ t \in \R$, and Theorem \ref{main2} is proved. \hfill $\square$

\section{Proof of Theorem \ref{mainPF}} \label{S5}

We will show  that conditions of Theorem \ref{mainPF} assure that each semi-wavefront  $u = \phi(x+ct)$, $(h,c)\in {\mathcal{D}}_{\frak L},$ of equation (\ref{17b}) is actually a monotone wavefront. Indeed, it is easy see 
that $0 < \phi(t) < \kappa,\ t \in \R,$ since otherwise, without loss of generality, we can  assume that $\phi(t_0)= \max_{s \in \R} \phi(s)$ for some $t_0$ that leads to the following contradiction:
$$\kappa \leq \phi(t_0)= \max_{s \in \R} \phi(s)=\int_\R N(t_0-s)g_1(\phi(s-ch))ds< \int_\R N(t_0-s)\max_{s\in \R}g_1(\phi(s))ds \leq \phi(t_0).$$
Next, we will need the following 
\begin{lemma} \label{v3} Set  $\Gamma(s):= g_1(\phi(s-ch))$. If the semi-wavefront $\phi(t)$ is increasing on $\R_-$ and satisfies
$\phi'(0) =0$  then, for  $t\in [0,ch]$, 
\begin{eqnarray}\nonumber
 \phi'(t) = \int_{-\infty}^0(N(t-s)-e^{\lambda_0(\xi)t} N(-s))d\Gamma(s)+  \int^{t}_0(N(t-s)-e^{\lambda_0(\xi)(t-s)}N(0))d\Gamma(s).
\end{eqnarray}
\end{lemma}
\begin{proof}  
Since $\Gamma(s)$ increases on $(-\infty, ch]$ and $\Gamma(-\infty)=0$, all Riemann-Stieltjes  integrals in the above formula are well defined and convergent. Next, note that  $g(s) = g_1(s)(1+\xi) - \xi s$ is of bounded variation on $[0,\kappa]$.  Thus, using \cite[Remark 9(2)]{TT} together with Corollary \ref{coco}, we conclude that $\phi(t)$ can have at most a finite number of critical points on each  interval $(-\infty, \alpha]$.  This implies that $\Gamma(s)$ has bounded variation on each $(-\infty, \alpha]$. 
Next, in view of Remark \ref{Re7},  after integrating by parts, we find  that
$$
\phi'(t)= \int_{-\infty}^tN'(t-s)\Gamma(s)ds+ \int^{+\infty}_t\lambda_0(\xi)N(0)e^{\lambda_0(\xi)(t-s)}\Gamma(s)ds=
$$
$$
 \int_{-\infty}^tN(t-s)d\Gamma(s)+ \int^{+\infty}_tN(0)e^{\lambda_0(\xi)(t-s)}d\Gamma(s)= $$
 $$\int_{-\infty}^tN(t-s)d\Gamma(s)+ e^{\lambda_0(\xi)t}N(0) \left(\int^{+\infty}_0e^{-\lambda_0(\xi)s}d\Gamma(s)-\int^{t}_0e^{-\lambda_0(\xi)s}d\Gamma(s)\right). 
$$
Since $\phi'(0)=0$, it holds  that 
 $$N(0)\int^{+\infty}_0e^{-\lambda_0(\xi)s}d\Gamma(s)=-\int_{-\infty}^0N(-s)d\Gamma(s)
$$
and therefore
$$
\phi'(t)= \int_{-\infty}^tN(t-s)d\Gamma(s)+ e^{\lambda_0(\xi)t} \left(-\int_{-\infty}^0N(-s)d\Gamma(s)-N(0)\int^{t}_0e^{-\lambda_0(\xi)s}d\Gamma(s)\right)=
$$
$$
\int_{-\infty}^0(N(t-s)-e^{\lambda_0(\xi)t} N(-s))d\Gamma(s)+  \int^{t}_0(N(t-s)-N(0)e^{\lambda_0(\xi)(t-s)})d\Gamma(s). \quad \square
$$
\end{proof}
\begin{thm} Assume   {\rm \bf(M)}, {\rm \bf(K)} and {\rm \bf(ST)}.  Then each semi-wavefront  $u = \phi(x+ct)$, $(h,c)\in {\mathcal{D}}_{\frak L},$ of equation (\ref{17b}) is a monotone wavefront. 
\end{thm}
\begin{proof} From \cite[Lemma 6]{TT}, we know that $\phi'(t)>0$ on some maximal interval $(-\infty, \sigma)$. If $\sigma=+\infty$, the corollary is proved. If $\sigma$ is finite, without loss of generality we may take 
$\sigma=0$.  Then $\Gamma(t) =
g_1(\phi(t-ch))$ is strictly increasing on $(-\infty, ch)$. But then Lemma \ref{v3} implies that $\phi'(t) \leq 0$ for all $t \in
(0,ch]$. Here, we are using the inequalities
$$
N(t-s) < N(-s) < e^{\lambda_0(\xi)t} N(-s), \ s \leq 0 <t; \ N(t-s) <N(0) <N(0)e^{\lambda_0(\xi)(t-s)},\  s< t.
$$
Thus $\phi'(t) \leq 0$ on some maximal interval $(0,\sigma_1)$. Note
that $\sigma_1$ must be a finite real number since otherwise
$\phi'(t) \leq  0$ on $(0,+\infty)$ implying $\phi(+\infty) =0$.
However, this contradicts the uniform persistence property of semi-wavefronts \cite{SEDY}. In consequence, $\sigma_1
> ch$ is finite so that $\phi'(\sigma_1) = 0, \
\phi''(\sigma_1) \geq 0$ and $\phi(\sigma_1) \leq 
\phi(\sigma_1-ch)$. On the other hand, we know that $\phi''(t) -c\phi'(t)-\phi(t) + g(\phi(t-ch))=0
$ for all  $t \in \R$ so that 
$$
\phi''(\sigma_1) - \phi(\sigma_1)+
g(\phi(\sigma_1-ch))=0,
$$
from which we obtain $\kappa > \phi(\sigma_1-h) \geq 
\phi(\sigma_1)\geq g(\phi(\sigma_1-h))>0$, a contradiction.  \hfill $\square$
\end{proof}

\section*{Acknowledgments}  \noindent   The authors express their  appreciation to
Professor Rafael Ortega whose insightful  suggestions helped to improve the original version of the paper.   
In particular, Subsection \ref{Sec22} is due to  these suggestions. 
 Especially we  would like to acknowledge  FONDECYT (Chile), project 1110309 for supporting  the research stay  of  Dr.  R. Ortega at the University of Talca.
This work was also supported by FONDECYT (Chile), 
projects 1120709 (E. Trofimchuk and M. Pinto) and 1150480 (S. Trofimchuk).

\end{document}